\documentclass{amsart}
\usepackage{amsmath}
\usepackage{amssymb}
\usepackage{amsfonts}
\usepackage{graphicx}
\usepackage{algorithm}
\usepackage{algorithmic}
\usepackage[hyperindex=true,bookmarks=true,bookmarksnumbered=true]{hyperref}

\newtheorem{defi}{Definition}
\newtheorem{thm}{Theorem}

\newcommand{\reals}{\mathbb{R}}

\title[Different shapes that are indistinguishable by the SRNF]{Closed surfaces with different shapes that are indistinguishable by the SRNF}
\author{Eric Klassen}
\address{Eric P. Klassen, Department of Mathematics, Florida State University, Tallahassee, FL, 32306 USA}
\email{klassen@math.fsu.edu}
\author{Peter W.~Michor}
\address{Peter W.~Michor, Fakult\"at f\"ur Mathematik,
Universit\"at Wien, Os\-kar-Mor\-gen\-stern-Platz 1, A-1090 Wien, Austria.}
\email{peter.michor@univie.ac.at.}

\thanks{Eric Klassen gratefully acknowledges the support of the Simons Foundation (Grant \# 317865).}
\begin{document}
\maketitle
\begin{abstract}
The Square Root Normal Field (SRNF), introduced by Jermyn et al. in \cite{jermyn2012SRNF}, provides a way of representing immersed surfaces in $\reals^3$, and equipping the set of these immersions with a ``distance function" (to be precise, a pseudometric) that is easy to compute. Importantly, this distance function is invariant under reparametrizations (i.e., under self-diffeomorphisms of the domain surface) and under rigid motions of $\reals^3$. Thus, it induces a distance function on the shape space of immersions, i.e., the space of immersions modulo reparametrizations and rigid motions of $\reals^3$. In this paper, we give examples of the degeneracy of this distance function, i.e., examples of immersed surfaces (some closed and some open) that have the same SRNF, but are not the same up to reparametrization and rigid motions. We also prove that the SRNF does distinguish the shape of a standard sphere from the shape of any other immersed surface, and does distinguish between the shapes of any two embedded strictly convex surfaces.
\end{abstract}

\section{Introduction}
It is a fascinating mathematical problem to distinguish shapes of different surfaces, and to quantify how different these shapes are from each other. One promising candidate for solving this problem has been the SRNF (square root normal field) method introduced by Jermyn et al. in \cite{jermyn2012SRNF} and expanded upon in \cite{jermyn2017}. Given an oriented surface $M$ with a Riemannian metric, this method introduces a pseudometric on the space 
$$\hbox{Imm}(M,\reals^3): =\{\hbox{immersions }M \to \reals^3\}.$$
This pseudometric is invariant under the right action of the group $\hbox{Diff}_+(M)$  by composition (where $\hbox{Diff}_+(M)$ denotes the group of orientation preserving diffeomorphisms $M\to M$), and is invariant under the left action of the group of rigid motions of $\reals^3$. Because of these invariances, this pseudometric on Imm$(M,\reals^3)$ induces a pseudometric on the shape space
$$\mathcal{S}= \hbox{Imm}(M,\reals^3)/(\hbox{Diff}_+(M)\times\hbox{rigid motions of }\reals^3).$$

The purpose of this paper is to give examples demonstrating that this pseudometric fails to be a metric on shape space for every domain $M$. Some of these examples were already known for open surfaces (or surfaces with boundary). We include them here anyway, because they have not appeared in the literature. However, the main contribution of this paper is to give a way of constructing examples for closed surfaces as well. The paper is structured as follows. Section \ref{Review} offers a brief review of the SRNF method. In Section \ref{counterexamples}, we give some examples of surfaces (open and closed) whose shapes are not determined by their SRNF. In Section \ref{spherespecial}, we prove that the shape of the sphere {\em is} determined by its SRNF. In Section \ref{convexspecial}, we prove that if two embedded, strictly convex, closed surfaces have the same SRNF, then they are translates of each other. In Section \ref{existproof}, we prove a theorem that is used in the presentation of the most important example in Section \ref{counterexamples}.

\section{Review of the Square Root Normal Field Method}\label{Review}
In this paper, we will assume that $M$ is a smooth, connected, oriented Riemannian 2-dimensional manifold with or without boundary. Let Imm$(M,\reals^3)$ denote the space of immersions of $M$ into $\reals^3$. Given $f\in\hbox{Imm}(M,\reals^3)$, the orientation on $M$ and the standard orientation on $\reals^3$ imply an orientation on the normal bundle of $f(M)$ at each point of $M$. Thus, given $f\in\hbox{Imm}(M,\reals^3)$, we can define $n:M\to\reals^3$ by letting $n(x)$ be the oriented unit normal vector to $f(M)$ at $f(x)$. Given such an $f$, we can also define $a:M\to \reals$ by $a(x)=$ the local area multiplication factor of $f$ at $x$. To give a precise formula for $a(x)$, let $\{v,w\}$ be an orthonormal basis of $T_xM$. Then let $a(x)=|df_x(v)\times df_x(w)|$.

The square root normal field (SRNF) map $\Phi:\hbox{Imm}(M,\reals^3)\to C^\infty(M,\reals^3)$ is defined by $\Phi(f)=q$, where
$$q(x)=\sqrt{a(x)}n(x).$$

We now discuss some of the important properties of the SRNF. First, it clearly induces a function 
$$\Phi:\hbox{Imm}(M,\reals^3)/\hbox{translations}\to C^\infty(M,\reals^3)$$
since $\Phi(f)$ depends only on the first derivative of $f$. Note that $\hbox{Diff}_+(M)$ acts on $\hbox{Imm}(M,\reals^3)$ by right composition. We define a right-action of $\hbox{Diff}_+(M)$ on $C^\infty(M,\reals^3)$ by $(g*\phi)(x)=\sqrt{b(x)}g(\phi(x))$, where $b(x)=$ the area multiplication factor of $\phi$ at $x$. Is easy to check that the map $\Phi$ is equivariant with respect to these two actions; i.e.
$$\Phi(f)*\phi=\Phi(f\circ\phi).$$
In addition, the action of $\hbox{Diff}_+(M)$ on $C^\infty(M,\reals^3)$ that we just defined is by linear isometries, if we put the usual $L^2$ inner product on $C^\infty(M,\reals^3)$. The obvious action of $SO(3)$ on $C^\infty(M,\reals^3)$ is also by isometries with respect to the $L^2$ inner product. (This action is defined by $A*q(x)=Aq(x)$).

These facts are important because we can use the SRNF to define a distance function on $\hbox{Imm}(M,\reals^3)$ by 
\begin{equation}\label{distfuncdef}
d(f_1,f_2)=\|\Phi(f_1)-\Phi(f_2)\|
\end{equation}
 where the norm refers to the $L^2$ norm. Define the {\em shape space} of immersions of $M$ in $\reals^3$ by
$${\mathcal S}=\hbox{Imm}(M,\reals^3)/(\hbox{Diff}_+(M)\times\hbox{rigid motions of }\reals^3 ).$$
In order to define a distance function on ${\mathcal S}$, we wish to define a distance function on $\hbox{Imm}(M,\reals^3)$ that is invariant under the actions of $\hbox{Diff}_+(M)$ and the group of rigid motions of $\reals^3$. The distance function defined above \eqref{distfuncdef} has these invariance properties! This is the main motivation for the use of the SRNF in shape analysis.

Though we called the function $d$ defined on $\hbox{Imm}(M,\reals^3)$ a ``distance function", there are some difficulties associated with that terminology; these difficulties are the main subject of this paper.

First, note that if two immersions differ only by a translation, their distance is zero. We don't consider that to be a serious problem, because we want to mod out by the group of rigid motions of $\reals^3$, which includes translations. 

However, it is easy to create examples of pairs of immersions that do not represent the same elements of shape space, but that still have distance zero with respect to $d$. We now give some of these examples in order of increasing seriousness.
\section{Counterexamples to the injectivity of the SRNF map on shape space}\label{counterexamples}

\subsection{Cylinders}

Let $M$ denote the cylinder in $\reals^3$ defined by $x^2+y^2=1$ and $0\leq z\leq 1$; we have, of course, the identity immersion $\hbox{Id}:M\to \reals^3$. Given $r>0$, define the linear transformation $L:\reals^3\to\reals^3$ by $L(x,y,z)=(rx,ry,\frac{z}{r})$. It is immediate that $L\circ \hbox{Id}$ is also an immersion, and that $\hbox{Id}$ and $L\circ \hbox{Id}$ both give rise to the same SRNF. Furthermore, it is obvious that these two immersions of the cylinder are not related by rigid motions or reparametrization. Of course, both of these immersions have Gaussian curvature 0 everywhere, making them rather special.

\subsection{Paraboloids}

Let $a$ and $b$ be non-zero real numbers and let $S_{a,b}$ denote the graph of the function $z=ax^2+by^2$ in $\reals^3$. 

\begin{thm}
If $ab=cd$, then we can parametrize $S_{a,b}$ and $S_{c,d}$ in such a way that they have the same SRNF.
\end{thm}

\begin{proof}
For a given $(a,b)$, parametrize $S_{a,b}$ by 
$$B(x,y)=\left({x\over a},{y\over b},{x^2\over a}+{y^2\over b}\right).$$ 
An easy computation then yields  
$$B_x\times B_y=\left(-\frac{2x}{ab},-\frac{2y}{ab}, \frac{1}{ab}\right).$$

Since $B_x\times B_y$ depends only on $ab$ (not on $a$ and $b$ individually), and the SRNF can be expressed as $\frac{B_x\times B_y}{\sqrt{|B_x\times B_y|}}$, the theorem follows.
\end{proof}

Once again, it is clear that if $a\neq \pm c$ and $b\neq\pm d$, then $S_{a,b}$ and $S_{c,d}$ are not equivalent under a rigid motion of $\reals^3$ or under reparametrization.

\subsection{Closed surfaces}\label{closedexamples}
\begin{defi} We define a {\em flat place} $D\subset \reals^2\times\{0\}$ to be a closed disc in $\reals^2\times\{0\}$ with $n$ disjoint open discs removed. These removed discs should be small enough that their boundaries are disjoint from each other and disjoint from the boundary of the original disc. Hence the boundary of a flat place consists of $n+1$ disjoint round circles; we call the largest of these, $C_0$, the {\em outer boundary component}, and the others $C_1,\dots, C_n$ the {\em inner boundary components}.
\end{defi}
Let $M\subset \reals^3$ be a closed, oriented, smooth surface with the following two properties. 
\begin{enumerate}
\item $M$ contains a flat place $D$ and
\item $D$ separates $M$ into $n+1$ components $M_0$, $M_1$, $\dots$,$M_n$, where $\partial M_i=C_i$ for each $i$. 
\end{enumerate} 
We have, of course, the identity immersion $\hbox{Id}:M\to \reals^3$. We now construct a new immersion $f:M\to\reals^3$.

First, choose a new flat place $\tilde D\subset \reals^2\times{0}$, whose outer boundary $\tilde C_0$ is the same as the outer boundary $C_0$ of $D$, while each of its inner boundary components $\tilde C_i$ is required to have the same radius as $C_i$, but may have a different location. For each $i=1,\dots,n$, let $T_i$ denote the translation $\reals^3\to\reals^3$ that takes $C_i$ to $\tilde C_i$. Now define $f:M\to\reals^3$ as follows:
\begin{itemize}
\item On $M_0$, $f(x)=x$
\item On each $M_i$ where $i>0$, $f(x)=T_i(x)$
\item On $D$, $f:D\to\tilde D$ is an area preserving diffeomorphism that is the identity on a neighborhood of $C_0$ and agrees with $T_i$ on a neighborhood of $C_i$ for each $i=1,\dots, n$. The existence of such an area-preserving diffeomorphism will be proved in Section \ref{existproof}, Theorem \ref{existdifftheorem}.
\end{itemize}
It is clear that $f$ is an area-preserving immersion, and also that for each $x\in M$, the unit normal vector of $M$ at $x$ is equal to the unit normal vector of $f(M)$ at $f(x)$. It follows immediately that $\Phi(\hbox{Id})=\Phi(f)$.

Here is an amusing way to visualize this situation. Imagine a round table-top with a chessboard painted on it, and with the 32 chess pieces sitting on the board. Consider the surface $M$ to be the boundary surface of this entire configuration: the tabletop together with the board and the pieces. If we move the chess pieces around on the board (without changing the orientation of each piece), the resulting surface is related to the original one as in the above construction. Thus, all surfaces obtained by moving around the chess pieces are indistinguishable using the SRNF.

\subsection{A closed example with a flip}

Here is one more example, similar to the chessboard example above. Let $M$ be a closed orientable surface in $\reals^3$ and assume we can write $M=M_0\cup D\cup M_1$ where $D=\{(x,y,z):z=0\hbox{ and }1\leq x^2+y^2\leq 4\}$, and $M_0$ intersects $D$ along its outer boundary and $M_1$ intersects $D$ along its inner boundary. Again, 
$\hbox{Id}:M\to \reals^3$ is an immersion of $M$ into $\reals^3$. We define another immersion $f:M\to\reals^3$ as follows:
\begin{itemize}
\item On $M_0$, $f(x,y,z)=(x,y,z)$
\item On $M_1$, $f(x,y,z)=(-x,-y,-z)$
\item On the annulus $D$, $f(x,y,0)=(x,y,0)$ on a neighborhood of the outer boundary, $f(x,y,0)=(-x,-y,0)$ on a neighborhood of the inner boundary, and $f$ rotates each concentric circle of $D$ by an angle that smoothly varies from $0$ radians on the outer boundary to $\pi$ radians on the inner boundary.
\end{itemize}
Once again, it is immediate that $\Phi(\hbox{Id})=\Phi(f)$, while these two immersions are not related by a reparametrization or by a rigid motion. In terms of the above visualization, we have turned one or more of the chess pieces upside down while simultaneously giving it a 180 degree horizontal rotation.
It is clear that the ``flat place" plays a crucial role in the last two examples. Thus there is an obvious question:

{\bf Question:} Are there examples of non-injectivity of the SRNF map on shape space if the domain is a closed orientable manifold, and the immersions don't have any flat places? 

\section {The sphere is distinguished from all other immersed surfaces by its SRNF}\label{spherespecial}

In this section, we prove that the standard round sphere is distinguished from every other immersion of $S^2$ by its SRNF. Let $S^2$ denote the unit sphere in $\reals^3$; we have the immersion Id$:S^2\to\reals^3$. 
\begin{thm}
If $f:S^2\to\reals^3$ is any immersion such that $\Phi(\hbox{Id})=\Phi(f)$, then $f=\hbox{Id}$ (up to translation in $\reals^3$).
\end{thm}
\begin{proof} Clearly, $x$ is the oriented unit normal to $S^2$ at $x$. Because of our assumption that $\Phi(\hbox{Id})=\Phi(f)$, it follows immediately that $x$ is also the oriented unit normal to $f(S^2)$ at $f(x)$. Therefore, locally, $f$ is the inverse of the Gauss map, which can be defined locally on $f(U)$, where $U$ is any open set in $S^2$ that is sufficiently small so that $f$ restricted to $U$ is injective. Since $\Phi(\hbox{Id})=\Phi(f)$, it follows that Id and $f$ have the same area multiplication factor at each point of $S^2$; hence they both have area multiplication factor of $1$ at each point. It follows immediately from the definition of Gaussian curvature that the area multiplication factor of the Gauss map at any point of a surface is equal to the absolute value of the Gaussian curvature of the surface at the same point. Hence, the Gaussian curvature of $f(S^2)$ at each point is equal to $\pm 1$. Since the Gaussian curvature of $f(S^2)$ is a continuous function of $S^2$, it follows that it is either 1 at every point or $-1$ at every point. By the Gauss Bonnet theorem, the Gaussian curvature of $f(S^2)$ is 1 at every point. It then follows that $f(S^2)$ is a unit sphere, by the classical theorem that any closed immersed surface in $\reals^3$ with constant Gaussian curvature $=1$ is a unit sphere in $\reals^3$.  (See, for example, \cite{LiebmannThm} for a recent statement and proof of this theorem for the case of immersed surfaces.)
\end{proof}

\section{A uniqueness result for convex embedded surfaces}\label{convexspecial}

\begin{thm}
Let $f_1, f_2:S^2\to\reals^3$ be embeddings such that $f_1(S^2)$ and $f_2(S^2)$ are both convex surfaces and both have Gaussian curvature strictly positive at every point. If $\Phi(f_1)=\Phi(f_2)$, then $f_1$ and $f_2$ differ by a translation.
\end{thm}
\begin{proof}
Since $f_1(S^2)$ is convex and has strictly positive Gaussian curvature, it follows that its Gauss map $G:f_1(S^2)\to S^2$ is a diffeomorphism.  If we replace both $f_1$ and $f_2$ by their compositions with $f_1^{-1}\circ G^{-1}$, the hypotheses of the theorem still hold. Hence, for the remainder of the proof, we may assume that $f_1$ is the inverse of the Gauss map $G:f_1(S^2)\to S^2$. Because $\Phi(f_1)=\Phi(f_2)$, it follows that $f_2$ is the inverse of the Gauss map $f_2(S^2)\to S^2$. For $i=1,2$, define $K_i:S^2\to\reals$ by $K_i(x)=$ the Gaussian curvature of $f_i(S^2)$ at $f_i(x)$. Because $\Phi(f_1)=\Phi(f_2)$, we know that $f_1$ and $f_2$ have the same area multiplication factor. But, by definition of Gaussian curvature, the area multiplication factor of the Gauss map is equal to the absolute value of the Gaussian curvature. Since $f_1(S^2)$ and $f_2(S^2)$ are both assumed to have positive Gaussian curvature, it follows that $K_1(x)=K_2(x)$ for every $x\in S^2$. According to the Minkowski \cite{Minkowski}, it follows that $f_1(S^2)$ and $f_2(S^2)$ differ only by a translation.

\end{proof}

\section{Proof of a theorem on area-preserving maps of a flat place}\label{existproof}

The goal of this section is to prove a theorem needed in Section \ref{closedexamples} of this paper. First, we need the following theorem. It remains true for manifolds with more general boundaries with the same proof;  for a formulation for Whitney manifold germs see \cite[end of Section 4]{Michor19a}.

\begin{thm}\label{MichorAreaPres} Let $M$ be a connected, 
compact, oriented manifold with corners with $\dim(M)=m$.
Let $\omega_0,\omega\in\Omega^{m}(M)$ be two 
volume forms (both $>0$) with $\int_M\omega_0=\int_M\omega$.
Suppose that there is a diffeomorphism $f:M\to M$ such that $f^*\omega |_U = \omega_0|_U$ for an open neighborhood of $\partial M$ in $M$. 

Then there exists a diffeomorphism $\tilde f:M\to M$ with $\tilde f^*\omega = \omega_0$ such that $\tilde f$ equals $f$ on an open neighborhood of $\partial M$.
\end{thm} 

\begin{proof} 
Choose an open neighborhood $V$ of $\partial M$ in $M$ 
with $V\subset \overline V\subset U$ and such that $M\setminus \overline V$ is connected.
Put $\omega_1:= f^*\omega$ and $\omega_t:=\omega_0+t(\omega_1-\omega_0)$ for $t\in[0,1]$; each $\omega_t$ 
is again a volume form.  

We search for a curve of diffeomorphisms $t\mapsto f_t$ with 
$f_t^*\omega_t=\omega_0$; then we should have $\frac{\partial}{\partial t}(f_t^*\omega_t)=0$.
Since $\int_M(\omega_1-\omega_0)=0$ and $\omega_1-\omega_0\in \Omega^m_c(M\setminus \overline V)$ has compact support and $\int_{M\setminus\overline V}: H^m_c(M\setminus\overline V)$ is an isomorphism, we have $[\omega_1-\omega_0]=0\in H^m_c(M\setminus\overline V)$, so 
$\omega_1-\omega_0=d\psi$ for some $\psi\in\Omega^{m-1}_c(M\setminus\overline V)$. 
Consider the time-dependent vector field $\eta_t:=(\frac{\partial}{\partial t}f_t)\circ f_t^{-1}$; then 
by \cite[31.11]{Michor08} we have:
\begin{align*}
0 &\overset{\text{wish}}{=} \tfrac{\partial}{\partial t}(f_t^*\omega_t) 
     = f_t^*\mathcal{L}_{\eta_t}\omega_t + f_t^*\tfrac{\partial}{\partial t}\omega_t 
     = f_t^*(\mathcal{L}_{\eta_t}\omega_t + \omega_1 - \omega_0), \\
0 &\overset{\text{wish}}{=} \mathcal{L}_{\eta_t}\omega_t + \omega_1 - \omega_0 
     = d i_{\eta_t}\omega_t + i_{\eta_t}d\omega_t + d\psi 
     = d i_{\eta_t}\omega_t + d\psi. 
\end{align*}
We can choose $\eta_t$ uniquely by $i_{\eta_t}\omega_t=-\psi$, since $\omega_t$ 
is nondegenerate for all $t$. Note that $\eta_t$ has compact support in $\operatorname{supp}(\psi)\subset M\setminus\overline V$.
Consequently,  the evolution operator 
$f_t=\Phi^\eta_{t,0}$ 
exists for $t\in [0,1]$, by \cite[3.30]{Michor08}. 
We have, using \cite[31.11.2]{Michor08}, 
\begin{equation*}
\tfrac{\partial}{\partial t}(f_t^*\omega_t) 
     = f_t^*(\mathcal{L}_{\eta_t}\omega_t + d\psi) 
     = f_t^*(d i_{\eta_t}\omega_t + d\psi) = 0,
\end{equation*}
so $t\mapsto f_t^*\omega_t = 
\omega_0$ is constant, $f_t|_{V} = \operatorname{Id}$ and 
$\omega_0= f_1^*\omega_1 = f_1^* f^* \omega = (f\circ f_1)^*\omega$. 
So $\tilde f = f\circ f_1$ is the solution.
\end{proof}

We  now state the theorem we used in Section \ref{closedexamples}:

\begin{thm}\label{existdifftheorem}
Let $D=\{(x,y)\in\reals^2:x^2+y^2\leq 1\}$. Let $D_1,\dots,D_n$ be closed discs contained in $D$ that are disjoint from each other and also disjoint from $\partial D$. Let $\tilde D_1,\dots,\tilde D_n$ be another collection of closed discs in $D$, also disjoint from each other and from $\partial D$. Assume that for each $i$, $D_i$ and $\tilde D_i$ have the same radius, and let $T_i$ be the translation of $\reals^2$ such that $T_i(D_i)=\tilde D_i$. Then there exists an area preserving diffeomorphism $D\to D$ that is the identity on a neighborhood of $\partial D$, and agrees with $T_i$ on a neighborhood of $D_i$ for each $i$.
\end{thm}
\begin{proof}
It follows from Theorem 3.2 on page 186 of Hirsch's {\em Differential Topology} \cite{HirschDiffTop} that there exists a diffeomorphism $F:D\to D$ such that $F$ is the identity on a neighborhood of $\partial D$, and agrees with $T_i$ on a neighborhood of each $D_i$. Theorem \ref{MichorAreaPres} implies that there is an area preserving diffeomorphism of $D$ that agrees with $F$ on a neighborhood of $\partial D$ and on a neighborhood of each $D_i$.

\end{proof}

\bibliographystyle{plain}
\bibliography{}

\end{document}